\documentclass[a4paper,twoside,11pt]{article}
\usepackage{a4wide, fancyhdr, amsmath, amssymb, mathtools, yfonts, mathrsfs}
\usepackage{graphicx}
\usepackage{cases}
\usepackage{tikz}
\usepackage[all]{xy}
\usepackage[utf8]{inputenc}
\usepackage{amsthm}
\usepackage[english]{babel}
\usepackage{chngcntr}
\usepackage{comment}
\usepackage{ifthen}
\usepackage{calc}
\usepackage{hyperref}
\usepackage[OT2,T1]{fontenc}
\usepackage{authblk}
\usepackage{multirow}
\usepackage{rotating}
\numberwithin{equation}{section}
\DeclareSymbolFont{cyrletters}{OT2}{wncyr}{m}{n}
\DeclareMathSymbol{\Sha}{\mathalpha}{cyrletters}{"58}

\providecommand{\keywords}[1]{{2010}\textit{ Mathematics Subject Classification. }#1}

\newcommand{\leg}[2]{\genfrac{(}{)}{}{}{#1}{#2}}

\newcommand{\ZZ}{\mathbb{Z}}
\newcommand{\QQ}{\mathbb{Q}}

\DeclareMathOperator{\CL}{Cl}

\newcommand\FF{\mathbb{F}_2}

\newcommand\DD{\mathcal{D}}
\newcommand\SD{\mathcal{S}}

\newcommand\UU{\mathcal{U}}

\newcommand\aaa{\mathfrak{a}}

\newcommand\uu{\mathbf{u}}
\newcommand\vv{\mathbf{v}}

\newcommand\dd{\mathfrak{d}}
\newcommand\pp{\mathfrak{p}}

\newcommand\PPP{\mathcal{P}}

\newcommand\pt{\mathfrak{t}}

\DeclareMathOperator{\rk}{rk}

\DeclareMathOperator{\Gal}{Gal}

\DeclareMathOperator{\Art}{Art}

\newcommand\NV{\mathcal{N}}
\newcommand\ve{\varepsilon}

\newcommand\cb{\mathbf{c}}

\newtheorem{theorem}{Theorem}
\newtheorem{lemma}{Lemma}[section]
\newtheorem{prop}[lemma]{Proposition}

\title{On Hasse's Unit Index}
\author{Djordjo Milovic\thanks{Gower Street, London, WC1E 6BT, United Kingdom, djordjo.milovic@ucl.ac.uk; supported by }}
\affil{Department of Mathematics, University College London}

\date{\today}

\begin{document}

\maketitle

\begin{abstract} 
We study the distribution of Hasse's unit index $Q(L)$ for the CM-fields $L = \QQ(\sqrt{d}, \sqrt{-1})$ as $d$ varies among positive squarefree integers. We prove that the number of $d\leq X$ such that $Q(L) = 2$ is proportional to $X/\sqrt{\log X}$. 
\end{abstract}

\section{Introduction}
Let $L$ be a CM biquadratic number field and let $K$ be its quadratic subfield. Hasse \cite{HasseBook} considered the unit index
$$
Q(L) = [U_{L}: U_{K}T_{L}],
$$
where $U_{K}$ and $U_{L}$ denote the unit groups of the rings of integers of $K$ and $L$, respectively, and $T_{L}$ denotes the torsion subgroup of $U_{L}$. In \cite[Theorem 1]{Lemmermeyer}, Lemmermeyer proved that if $|T_L| = 4$, then $Q(L)\in\{1, 2\}$ and 
$$
Q(L) = 2\Longleftrightarrow 2\text{ ramifies in }K\text{ and the prime of }K\text{ lying above }2\text{ is principal.}
$$
The biquadratic CM fields $L$ satisfying $|T_L| = 4$ are in one-to-one correspondence with squarefree integers $d>1$, with the correspondence given by
$$
d\longleftrightarrow \QQ(\sqrt{d}, \sqrt{-1}).
$$ 
The prime $2$ ramifies in $K = \QQ(\sqrt{d})$ if and only if $d\not\equiv 1\bmod 4$. In this case, the prime of $K$ lying above $2$ is principal only if there exist integers $x$ and $y$ such that
\begin{equation}\label{eq:Legendre2}
x^2 - dy^2 = \pm 2,
\end{equation}
which can occur only if $\pm 2$ is a square modulo every odd prime $p$ dividing $d$. Hence, with $L = \QQ(\sqrt{d}, \sqrt{-1})$, we have
$$
Q(L) = 2 \Longrightarrow d\in \DD_{2} \text{ or }d\in \DD_{-2},
$$
where
$$
\DD_{2} = \{d>1\text{ squarefree and}\not\equiv 1\bmod 4: p\text{ prime dividing }d\Rightarrow p\not\equiv 3, 5\bmod 8\}
$$
and
$$
\DD_{-2} = \{d>1\text{ squarefree and}\not\equiv 1\bmod 4: p\text{ prime dividing }d\Rightarrow p\not\equiv 5, 7\bmod 8\}.
$$
These sets are analogous to the set of \textit{special discriminants} appearing in the work of Fouvry and Kl\"{u}ners in the context of the negative Pell equation \cite{FK2}. For a subset $\Omega$ of the natural numbers and a real number $X>0$, we will write $\Omega(X)$ for the set of $n\in\Omega$ such that $n\leq X$. As $X\rightarrow\infty$, we have
$$
|\DD_{2}(X)| \sim \frac{2C_{2}}{3}\frac{X}{\sqrt{\log X}}
$$
and
$$
|\DD_{-2}(X)| \sim \frac{2C_{-2}}{3}\frac{X}{\sqrt{\log X}},
$$
where $C_2$ and $C_{-2}$ are positive real numbers defined in \eqref{eq:constants}. Setting
$$
\SD = \{d>1\text{ squarefree}: Q(L) = 2\},
$$
where as before $L = \QQ(\sqrt{d}, \sqrt{-1})$, we immediately deduce that
$$
|\SD(X)| \ll \frac{X}{\sqrt{\log X}}.
$$
Our main goal is to give a relatively simple proof that $|\SD(X)| \gg \frac{X}{\sqrt{\log X}}$.
\begin{theorem}\label{MainThm}
As $X\rightarrow \infty$, we have
$$
c_1\frac{X}{\sqrt{\log X}}\left(1-o(1)\right)\leq|\SD(X)| \leq c_2\frac{X}{\sqrt{\log X}}\left(1+o(1)\right),
$$
where 
$$
c_1 = \frac{C_{2}}{6}\prod_{j = 1}^{\infty}(1-2^{-j}) > 0
$$
and
$$
c_2 = \frac{2C_{2} + 2C_{-2}}{3}.
$$
\end{theorem}
Our proof relies on computing the distribution of the $4$-rank of narrow class groups $\CL^+(8d)$ of the real quadratic fields $\QQ(\sqrt{2d})$ for $2d\in\DD_2$. For $2d\in \DD_2$, the $4$-rank of these groups turns out to be substantially larger on average than the $4$-rank of narrow class groups of real quadratic fields; compare \cite[(7), p.\ 458]{FK1} to \eqref{eq:asymptotic2}. As a result, unlike in the case of generic real quadratic fields, it is not possible to deduce a positive proportion of $2d\in\DD_2$ with $4$-rank of $\CL^+(8d)$ equal to $0$ by simply studying the first moment of the $4$-rank. We thus compute the full distribution of the $4$-rank of $\CL^+(8d)$ via the method of moments developed by Fouvry and Kl\"{u}ners \cite{FK1}. The implementation of this method to the family of $2d\in\DD_2$ necessitates a new combinatorial argument. 

For stronger results on the Hasse unit index in the context of certain thin families of discriminants, see \cite[Corollary 3, p.\ 2]{M4}. Finally, although our work concerns the same biquadratic fields as those appearing in the recent work \cite{FKO}, our results have been developed independently.

\section{Algebraic preliminaries}
\subsection{Criteria for solvability over $\ZZ$}
Given a fundamental discriminant $D$, let $\CL^+(D)$ denote the narrow class group of the quadratic number field $\QQ(\sqrt{D})$. We will denote the group operation in $\CL^+(D)$ by multiplication, as it is induced by multiplication of ideals. The Artin map gives a canonical isomorphism
$$
\Art: \CL^+(D)\longrightarrow\Gal(H_D/\QQ(\sqrt{D})),
$$
where $H_D$ is the maximal abelian extension of $\QQ(\sqrt{D})$ that is unramified at all finite primes. Between $\QQ(\sqrt{D})$ and $H_Dn$ lies the genus field of $\QQ(\sqrt{D})$, which we denote by $G_D$. It is the subfield of $H_D$ fixed by the image of the squares in the narrow class group, i.e.,
$$
G_D = H_D^{\Art(\CL^+(D)^2)},
$$
and it can also be characterized as the maximal abelian extension of $\QQ$ contained in $H_D$. Being able to compute the restriction of the image of the Artin map to $G_D$ allows us to check if a given ideal class in $\CL^+(D)$ is a square. Indeed, if we denote the class of an ideal $\aaa$ by $[\aaa]$, then we have
\begin{equation}\label{eq:checkifsquare}
[\aaa] \in \CL^+(D)^2\Longleftrightarrow \Art(\aaa)|_{G_D} = 1.
\end{equation}
Let $p_1, \ldots, p_t$ denote the primes dividing $D$, with $p_2, \ldots, p_t$ odd. Then $G_D$ can be generated as the mutiquadratic field
$$
G_D = \QQ(\sqrt{D}, \sqrt{p_2^{\ast}}, \cdots, \sqrt{p_t^{\ast}}),
$$
where, for an odd prime $p$, we write 
$$
p^{\ast} = (-1)^{(p-1)/2}p = \leg{-1}{p}p.
$$
Here and henceforth, $\leg{\cdot}{\cdot}$ denotes the Jacobi symbol. Note that
\begin{equation}\label{eq:genus}
\Gal(G_D/\QQ(\sqrt{D}))\cong \FF^{t-1}\quad\text{ and }\quad[G_D : \QQ(\sqrt{D})] = 2^{t-1}.
\end{equation}

The $2$-torsion subgroup of $\CL^+(D)$ is generated by the classes of the prime ideals $\pp_1, \ldots, \pp_t$ lying above $p_1, \ldots, p_t$, respectively. For each $\mathbf{e} = (e_1, \ldots, e_t)\in\FF^t$, we define the ideal $\aaa_{\mathbf{e}} = \pp_1^{e_1}\cdots \pp_t^{e_t}$, and we write $a_{\mathbf{e}}$ for the absolute norm of $\aaa_{\mathbf{e}}$. By \eqref{eq:genus}, there exists a unique non-zero $\mathbf{e}\in\FF^t$ such that the class of $\aaa_{\mathbf{e}}$ is trivial. Supposing for simplicity that $D\equiv 0\bmod 4$ and setting $d = D/4$, we see that the equation $x^2 - dy^2 = a_{\mathbf{e}}$ is solvable over $\ZZ$ for exactly one non-zero $\mathbf{e}\in\FF^t$. Since $a_{\mathbf{e}}$ varies over the positive squarefree divisors of $D$ as $\mathbf{e}$ varies over $\FF^t$, we see that there exists exactly one squarefree integer $a>1$ such that $a$ divides $D$ and such that $x^2 - dy^2 = a$ is solvable over $\ZZ$. 

Now suppose that $D = da$ or $D = 4da$, where $d$ and $a$ are coprime positive squarefree integers, and consider the equation
\begin{equation}\label{eq:aplus}
x^2 - da y^2 = a.
\end{equation}
Let $\aaa$ be the unique ideal of the form $\aaa_{\mathbf{e}}$ as above of absolute norm $a$. Then \eqref{eq:aplus} is solvable over $\ZZ$ if and only if the class of $\aaa$ is trivial in $\CL^+(D)$. Suppose that the class of $\aaa$ is a square in $\CL^+(D)$. Since $[\aaa]$ is a $2$-torsion element in $\CL^+(D)$, this indicates the existence of an element of order $4$ in $\CL^+(D)$, unless of course $[\aaa]$ is trivial. Therefore, \textit{if} $[\aaa]\in\CL^+(D)^2$ \textit{and} $\CL^+(D)$ has no elements of order $4$, \textit{then} $[\aaa]$ must be trivial and hence the equation \eqref{eq:aplus} must be solvable over $\ZZ$.

Similarly, consider the equation 
\begin{equation}\label{eq:aminus}
x^2 - da y^2 = -a,
\end{equation}
solvable over $\ZZ$ if and only if $x^2 - da y^2 = d$ is solvable over $\ZZ$. Let $\dd$ be the unique ideal of the form $\aaa_{\mathbf{e}}$ as above of absolute norm $d$. Similarly as above, \textit{if} $[\dd]\in\CL^+(D)^2$ \textit{and} $\CL^+(D)$ has no elements of order $4$, \textit{then} $[\dd]$ must be trivial and hence the equation \eqref{eq:aminus} must be solvable over $\ZZ$.

With $a = 2$, these observations lead us to the following propositions for the solvability of \eqref{eq:Legendre2}.
\begin{prop}\label{Prop2}
Let $d$ be a positive odd squarefree integer. Suppose that
\begin{itemize}
\item the narrow class group of $\QQ(\sqrt{2d})$ has no elements of order $4$, and
\item $\leg{2}{p} = 1$ for all primes $p$ dividing $d$.
\end{itemize}
Then the equation
$$
x^2-2d y^2 = 2
$$
has a solution in integers $x$ and $y$.
\end{prop}
\begin{proof}
Suppose that $d = p_1\cdots p_r$. Let $\pp_i$ (resp.\ $\pt$) denote the ideal of $\ZZ[\sqrt{2d}]$ lying above $p_i$ (resp.\ $2$). Let $D = 8d$ be the discriminant of  $\QQ(\sqrt{2d})$. By the first assumption and the observation above, it suffices to show that in each case we have $[\pt]\in \CL^+(D)^2$. By \eqref{eq:checkifsquare}, $[\pt]\in \CL^+(D)^2$ if and only if 
$$
\Art(\pt)|_{G_D} = 1.
$$
We have 
$$
G_{D} = \QQ(\sqrt{2\ve}, \sqrt{p_1^{\ast}}, \sqrt{p_2^{\ast}}, \ldots, \sqrt{p_r^{\ast}}),
$$
where
$$
\ve = (-1)^{(d-1)/2}\in\{\pm 1\}.
$$
The prime $\pt$ splits in $\QQ(\sqrt{2d}, \sqrt{2\ve}) $ if and only if $d\ve\equiv \pm 1\bmod 8$, i.e., if and only if $d\equiv \pm 1\bmod 8$. Moveover, $\pt$ splits in $\QQ(\sqrt{2d}, \sqrt{p_i^{\ast}}) $ if and only if $p_i^{\ast}\equiv \pm 1\bmod 8$, i.e., if and only if $p_i\equiv \pm 1\bmod 8$. Thus $\Art(\pt)|_{G_D}$ can be viewed as the element 
$$
\left(\leg{2}{d}, \leg{2}{p_1}, \ldots, \leg{2}{p_r}\right)
$$
of 
$$
\Gamma = \Gal(G_D/\QQ(\sqrt{D})) \cong \{(a_1, \ldots, a_{r+1})\in \{\pm 1\}^{r+1}: a_1\cdots a_{r+1} = 1\}.
$$
Hence $[\pt]\in \CL^+(D)^2$ if and only if 
$$
\leg{2}{p_i} = 1\quad\text{ for all }i.
$$
\end{proof}

\begin{prop}\label{PropMinus2}
Let $d$ be a positive odd squarefree integer. Suppose that
\begin{itemize}
\item the narrow class group of $\QQ(\sqrt{2d})$ has no elements of order $4$, and
\item $\leg{-2}{p} = 1$ for all odd primes $p$ dividing $d$. 
\end{itemize}
Then the equation
$$
x^2-2d y^2 = -2
$$
has a solution in integers $x$ and $y$.
\end{prop}
\begin{proof}
Let $p_i$, $\pp_i$, $\pt$, $D$, $r$, $\ve$, and $\Gamma$ be as in the proof of Proposition~\ref{Prop2}. Let $\dd = \pp_1\cdots \pp_r$. By the first assumption and the observation above, it suffices to show that in each case we have $[\dd]\in \CL^+(D)^2$. By \eqref{eq:checkifsquare}, $[\dd]\in \CL^+(D)^2$ if and only if 
$$
\Art(\dd)|_{G_D} = \Art(\pp_1)|_{G_D} \cdots \Art(\pp_r)|_{G_D} = 1.
$$
The Artin symbol $\Art(\pp_i)|_{G_D}$ viewed as an element of $\Gamma$ is equal to 
$$
\left(\leg{2\ve}{p_i}, \leg{p_1^{\ast}}{p_i}, \ldots, \leg{p_{i-1}^{\ast}}{p_i}, \leg{2d/p_i^{\ast}}{p_i}, \leg{p_{i+1}^{\ast}}{p_i}, \ldots, \leg{p_r^{\ast}}{p_i}\right).
$$ 
Then, by multiplicativity of the Artin symbol, we have
$$
\Art(\dd)|_{G_D} = \left(\leg{2\ve}{d}, \leg{p_1^{\ast}}{d/p_1}\leg{2d/p_1^{\ast}}{p_1}, \ldots, \leg{p_r^{\ast}}{d/p_r}\leg{2d/p_r^{\ast}}{p_r}\right)\in\Gamma.
$$
Hence $[\dd]\in \CL^+(D)^2$ if and only if 
$$
\leg{2\ve}{d} = 1
$$
and
\begin{equation}\label{eq:allforone}
1 = \leg{p_i^{\ast}}{d/p_i}\leg{2d/p_i^{\ast}}{p_i} = \leg{d/p_i}{p_i}\leg{2d/p_i^{\ast}}{p_i} = \leg{2}{p_i}\leg{-1}{p_i} = \leg{-2}{p_i} \text{ for all }i, 
\end{equation}
where the second equality follows from the law of quadratic reciprocity. Of course, as the product of the $r+1$ entries of $\Art(\dd)|_{G_D}\in \Gamma$ is equal to $1$, the condition \eqref{eq:allforone} implies that $\leg{2\ve}{d} = 1$.
\end{proof}

\subsection{Formula for the $4$-rank}
Let $D$ be a fundamental discriminant. In \cite{FK1}, Fouvry and Kl\"{u}ners give the following formula for the $4$-rank of $\CL^+(D)$, i.e., for the quantity
$$
|\CL^+(D)^2/\CL^+(D)^4| =: 2^{\rk_4\CL^+(D)}.
$$
We are interested in the fields $\QQ(\sqrt{2d})$. Writing $D = 8d$ with $d$ odd and squarefree, Fouvry and Kl\"{u}ners \cite[(111), p.\ 503]{FK1} prove that
\begin{equation}\label{eq:4rank2}
2^{\rk_4\CL^+(8d)} = \frac{1}{2\cdot 2^{\omega(d)}}\sum_{d=D_{0}D_{1}D_{2}D_{3}}\leg{2}{D_{3}}\leg{D_{2}}{D_{0}}\leg{D_{1}}{D_{3}}\leg{D_{3}}{D_{0}}\leg{D_{0}}{D_{3}}\left[\leg{-1}{D_{0}}+\leg{-1}{D_{3}}\right].
\end{equation}
Here the sum is over $4$-tuples of positive integers $(D_0, D_1, D_2, D_3)$ such that $D_0D_1D_2D_3 = d$. Now suppose that $2d\in\DD_2$, so that every prime $p$ dividing $d$ satisfies $\leg{2}{p} = 1$. Then the factor $\leg{2}{D_3}$ in the right-hand side of \eqref{eq:4rank2} is trivial; furthermore, the symmetry in $D_0$ and $D_3$ in the ensuing formula allows us to rewrite \eqref{eq:4rank2} for $2d\in\DD_2$ as
\begin{equation}\label{eq:4rank2v2}
2^{\rk_4\CL^+(8d)} = \frac{1}{2^{\omega(d)}}\sum_{d=D_{0}D_{1}D_{2}D_{3}}\leg{-1}{D_{3}}\leg{D_{2}}{D_{0}}\leg{D_{1}}{D_{3}}\leg{D_{3}}{D_{0}}\leg{D_{0}}{D_{3}},
\end{equation}
where now the sum is over $4$-tuples of positive integers $(D_0, D_1, D_2, D_3)$ such that $D_0D_1D_2D_3 = d\in \DD_2$.

Relabelling the indices in \eqref{eq:4rank2v2} as in \cite{FK1} by converting them into their binary expansions, so that $D_{0}$ becomes $D_{00}$, $D_{1}$ becomes $D_{01}$, etc., we obtain
\begin{equation}\label{eq:4rank2v3}
2^{\rk_4\CL^+(8d)} = \frac{1}{2^{\omega(d)}}\sum_{d=D_{00}D_{01}D_{10}D_{11}}\left(\prod_{\uu\in\FF^2}\leg{-1}{D_{\uu}}^{\lambda_1(\uu)}\right)\left(\prod_{(\uu, \vv)\in\FF^4}\leg{D_{\uu}}{D_{\vv}}^{\Phi_1(\uu, \vv)}\right),
\end{equation}
where $\lambda_1$ is the $\FF$-valued function defined by
\begin{equation}\label{eq:deflambda1}
\lambda_1(\uu) = u_1u_2, \quad \uu = (u_1, u_2).
\end{equation}
We will compute the average of $k$th moments of $2^{\rk_4\CL^+(8d)}$ as $2d$ varies among elements of $\DD_2$ such that $d\equiv 3\bmod 4$. We remark that when $2d\in\DD_2$ and $d\equiv 1\bmod 4$, then $\rk_{4}\CL^+(8d)\geq 1$, and so Proposition~\ref{Prop2} cannot be applied for such $d$.

Raising both sides of \eqref{eq:4rank2v3} to the $k$th power and decomposing the summation variables into products of their mutual greatest common denominators, as in \cite[(22), p.\ 471]{FK1}, we obtain the following analogue of \cite[Lemma 28, p.\ 493]{FK1}:
\begin{equation}\label{eq:total4rank2total}
S(X, k; 3, 4) := \sum_{\substack{2d\in\DD_{2}(X) \\ d\equiv 3\bmod 4}}2^{k\rk_4\CL^+(8d)} = \sum_{(D_{\uu})}\left(2^{-k\omega(D_{\uu})}\right)\left(\prod_{\uu}\leg{-1}{D_{\uu}}^{\lambda_k(\uu)}\right)\prod_{\uu, \vv}\leg{D_{\uu}}{D_{\vv}}^{\Phi_k(\uu, \vv)},
\end{equation}
where the sum is over $4^k$-tuples $(D_{\uu})$ of integers $D_{\uu}\in \DD_2$ indexed by elements $\uu$ of $\FF^{2k}$ and satisfying
$$
\prod_{\uu\in\FF^{2k}}D_{\uu}\leq X/2,\quad \prod_{\uu\in\FF^{2k}}D_{\uu}\equiv 3\bmod 4;
$$
where the function $\lambda_{k}: \FF^{2k}\rightarrow \FF$ is defined by
\begin{align}
\lambda_{k}(\uu) & = \sum_{j = 0}^{k-1}u_{2j+1}u_{2j+2};
\end{align}
and where the function $\Phi_{k}: \FF^{2k}\times \FF^{2k}\rightarrow \FF$ is defined by
\begin{align}
\Phi_{k}(u_1, \ldots, u_{2k}, v_1, \ldots, v_{2k}) & = \Phi_1(u_1, u_2, v_1, v_2) + \cdots + \Phi_1(u_{2k-1}, u_{2k}, v_{2k-1}, v_{2k}) \\
& = (u_1+v_1)(u_1+v_2)+\ldots +(u_{2k-1}+v_{2k-1})(u_{2k-1}+v_{2k}).
\end{align}

\section{Analytic preliminaries}
We first state the asymptotic formulas for the size of $\DD_{\pm 2}(X)$. Recall that we defined the sets
$$
\DD_{\pm 2} = \{d>1\text{ squarefree and}\not\equiv 1\bmod 4: \leg{\pm 2}{p} = 1\text{ for all odd primes }p|d\}.
$$
Define the sets
$$
\PPP_{\pm 2} = \{p\text{ prime number}:\text{if }p\text{ is odd, then }\leg{\pm 2}{p} = 1\}.
$$
and define the positive constants 
\begin{equation}\label{eq:constants}
C_{\pm 2} = \frac{1}{\sqrt{\pi}}\lim_{s\rightarrow 1}\left(\sqrt{s-1}\prod_{p\in\PPP_{\pm 2}}\left(1+\frac{1}{p^s}\right)\right).
\end{equation}
Then, similarly as in \cite{Park}, one can use results from \cite{Tenenbaum} to deduce that
\begin{equation}\label{asymptotic2}
\DD_{\pm 2}(X) = \frac{2C_{\pm 2}}{3}\frac{X}{\sqrt{\log X}} + O\left(\frac{X}{(\log X)^{\frac{3}{2}}}\right),
\end{equation}
where the implied constant is absolute. Again as in \cite{Park}, we can refine these formulas by restricting to congruence classes. The particular asymptotic formula that we need is
$$
A(X; 3, 4) := |\{2d\in \DD_2:\ d\leq X/2, d\equiv 3\bmod 4\}| \sim \frac{C_{2}}{6}\frac{X}{\sqrt{\log X}}.
$$
The treatment of $S(X; 2; 3, 4)$, which is by now standard due to the work of Fouvry and Kl\"{u}ners \cite{FK1, FK2, FK3} and Park \cite{Park}, proceeds in several steps, culminating in the following formula analogous to \cite[Proposition 5, p.\ 483]{FK1}:
\begin{equation}\label{eq:asymptotic}
S(X, k; 3, 4) = A(X; 3, 4)\cdot 2^{1-2^k}\sum_{\UU}\gamma^+(\UU, 1) + O_{\epsilon}\left(X(\log X)^{-\frac{1}{2}-\frac{1}{2^{k+1}}+\epsilon}\right)
\end{equation}
where the sum is over all maximal unliked subsets $\UU$ of $\FF^{2k}$ and where $\gamma^+(\UU)$ is defined as in \cite[(81), p.\ 493]{FK1}, i.e.,
$$
\gamma^+(\UU, 1) = \sum_{(h_{\uu})}\left(\prod_{\uu\in \UU}(-1)^{\lambda_k(\uu)\cdot \frac{h_{\uu}-1}{2}}\right)\left(\prod_{\uu, \vv}(-1)^{\Phi_k(\uu, \vv)\cdot \frac{h_{\uu}-1}{2}\cdot \frac{h_{\vv}-1}{2}}\right),
$$
where the sum is over $(h_{\uu})_{\uu\in\UU}\in\{\pm 1\bmod 4\}^{2^k}$ satisfying $\prod_{\uu\in\UU}h_{\uu}\equiv 3\bmod 4$ and where the last product is over unordered pairs $\{\uu, \vv\}\subset \UU$. We recall that two indices $\uu$ and $\vv$ in $\FF^{2k}$ are said to be unlinked if $\Phi(\uu, \vv) = \Phi(\vv, \uu)$. 

We will now say a few words about the derivation of the formula \eqref{eq:asymptotic}. First, the sum $S(X, k; 3, 4)$ is estimated by (i) bounding the contribution of terms where the number of prime factors of $D_{\uu}$ is too large, as in \cite[Section 5.3]{FK1}, (ii) partitioning the tuples $(D_{\uu})$ into ``diadic'' boxes of reasonable size, as in \cite[Section 5.4, p.\ 474-475]{FK1}, (iii) bounding the contribution from boxes featuring ``double oscillation'' of characters, as in \cite[Section 5.4, p.\ 476]{FK1}, and (iv) bounding the contribution from boxes featuring linked variables of vastly different sizes, as in  \cite[Section 5.4, p.\ 476-478]{FK1}. Once this is accomplished, we arrive at an analogue of \cite[Proposition 3, p.\ 479]{FK1}, at which point we partition the sum according to the congruence classes of $D_{\uu}$ modulo $4$. We then use quadratic reciprocity to pull out the factor $\sum_{\UU}\gamma^+(\UU, 1)$.  Via a variant of the prime number theorem, as in \cite[Lemma 19, p.\ 480]{FK1} or \cite[Lemma 6.3, p.\ 21]{Park}, we then remove the congruence conditions on the $D_{\uu}$ modulo $4$, which recovers $A(X; 3, 4)$ but comes with the cost of a factor of $2^{1-2^k}$ (as $\prod_{\uu}D_{\uu}\equiv 3\bmod 4$, fixing $2^k-1$ of the variables $D_{\uu}$ modulo $4$ determines the remaining variable modulo $4$).

\section{Combinatorics of the coefficient of the main term}
In this section, we analyze the coefficients of the main terms in the asymptotic formulas for the $k$th moments of the $4$-rank. 

As in \cite[(87), p.\ 494]{FK1}, for $\nu\in\FF$, we let
$$
\gamma^+(\UU, \nu) = \sum_{\substack{S\subset \UU \\ s\equiv \nu\bmod 2}}(-1)^{e^+(S)},
$$
where 
$$
e^+(S) = \sum_{\uu\in S}\lambda_k(\uu) + \sum_{\uu, \vv}\Phi(\uu, \vv),
$$
where the last sum is over unordered pairs $\{\uu, \vv\}\subset \UU$. This generalizes the quantity $\gamma^+(\UU, 1)$, which we aim to compute. The argument in \cite[Section 6]{FK1} culminates in the formula \cite[(105), p.\ 499]{FK1}
$$
\sum_{\UU}\gamma^+(\UU, 0) = 2^{2^k-1}\left(\NV(k+1, 2)- \NV(k, 2)\right).
$$
Here, as in \cite{FK1}, $\NV(m, 2)$ denotes the total number of $\FF$-vector subspaces of $\FF^m$. We will now adapt this argument to show that
\begin{equation}\label{eq:maincoeff2}
\sum_{\UU}\gamma^+(\UU, 1) = 2^{2^k-1}\NV(k, 2).
\end{equation}
The same argument as that for \cite[(101), p.\ 498]{FK1} yields the formula
$$
\sum_{\UU}\gamma^+(\UU, 1) = 2^{-k}\sum_{\UU_0\text{ good}}\sum_{\cb\in\FF^{2k}}\sum_{\substack{S\subset\cb+\UU_0\\ s\text{ odd}}}(-1)^{\lambda_k(\sigma)}.
$$ 
Consider the group homomorphism 
\begin{equation}\label{eq:defmapmu}
\mu: \PPP(\cb+\UU_0)\rightarrow\UU_0\oplus\left(\cb+\UU_0\right)
\end{equation}
given by
$$
\mu(S) = \sigma.
$$
We check that $\mu$ is surjective. First, note that $\mu(\emptyset) = 0$. If $\sigma\in\UU_0\setminus\{0\}$, then $\mu(\{\cb, \cb+\sigma\}) = \sigma$. Finally, if $\cb+\sigma\in\cb+\UU_0$, then $\mu(\{\cb+\sigma\}) = \cb+\sigma$. Hence, counting the cardinalities of the domain and the codomain of $\mu$, we find that each fiber of $\mu$ has cardinality $2^{2^k}/2^{k+1} = 2^{2^k-k-1}$. Also note that the image of $\PPP_1(\cb+\UU_0)$ under $\mu$ is exactly $\cb+\UU_0$. This implies that
$$
\sum_{\UU}\gamma^+(\UU, 1) = 2^{2^k-2k-1}\sum_{\UU_0\text{ good}}\sum_{\cb\in\FF^{2k}}\sum_{\sigma\in\cb+\UU_0}(-1)^{\lambda_k(\sigma)}.
$$
Since $\UU_0$ is a $k$-dimensional vector subspace of $\FF^{2k}$, for each coset $\UU$ of $\UU_0$ in $\FF^{2k}$, there are exactly $2^k$ vectors $\cb\in\FF^{2k}$ such that $\UU = \cb+\UU_0$. Hence
\begin{align*}
\sum_{\UU}\gamma^+(\UU, 1) &= 2^{2^k-k-1}\sum_{\UU_0\text{ good}}\sum_{\substack{\UU\\\text{coset of }\UU_0}}\sum_{\sigma\in\UU}(-1)^{\lambda_k(\sigma)} \\
& = 2^{2^k-k-1}\sum_{\UU_0\text{ good}}\sum_{\sigma\in\FF^{2k}}(-1)^{\lambda_k(\sigma)} \\
& = 2^{2^k-k-1}\NV(k, 2)\sum_{\sigma\in\FF^{2k}}(-1)^{\lambda_k(\sigma)}.
\end{align*}
It remains to compute $\sum_{\sigma\in\FF^{2k}}(-1)^{\lambda_k(\sigma)}$. Recall that $\lambda_k(\sigma) = \sum_{j = 0}^{k-1}\sigma_{2j+1}\sigma_{2j+2}$. Let
\begin{equation}\label{eq:lambdamod2}
m(\sigma) = \#\{j\in\{0, \ldots, k-1\}: \sigma_{2j+1} = \sigma_{2j+2} = 1\},
\end{equation}
so that $\lambda_k(\sigma)\equiv m(\sigma)\bmod 2$. For each $m\in \{0, \ldots, k\}$, the number of $\sigma\in\FF^{2k}$ such that $m(\sigma) = m$ is equal to 
$$
\binom{k}{m}\cdot 3^{k-m},
$$
since for each of the $k-m$ choices of $j$ for which $(\sigma_{2j+1}, \sigma_{2j+2})\neq (1, 1)$, the pair $(\sigma_{2j+1}, \sigma_{2j+2})$ can take one of three possible values, namely $(0, 0)$, $(0, 1)$, and $(1, 0)$. Thus
$$
\sum_{\sigma\in\FF^{2k}}(-1)^{\lambda_k(\sigma)} = \sum_{m = 0}^k\binom{k}{m}\cdot 3^{k-m}\cdot(-1)^{m} = (3-1)^{k} = 2^k,
$$
by the binomial theorem. In conclusion, we find that
$$
\sum_{\UU}\gamma^+(\UU, 1) = 2^{2^k-k-1}\NV(k, 2)\cdot 2^k = 2^{2^k-1}\NV(k, 2),
$$
as was claimed in \eqref{eq:maincoeff2}.

\section{Conclusion of the proof of Theorem~\ref{MainThm}}
Substituting \eqref{eq:maincoeff2} into \eqref{eq:asymptotic}, we obtain the asymptotic
\begin{equation}\label{eq:asymptotic2}
S(X, k; 3, 4) \sim \NV(k, 2)\cdot A(X; 3, 4).
\end{equation}
Note that the average of the $k$th moment for $2d\in\DD_2$ with $d\equiv 3\bmod 4$ is the same as the average of the $k$th moment for general negative discriminants; see \cite[Equations (4), (6), (8)]{FK1}. Putting the formula \eqref{eq:asymptotic2} for the $k$th moments through the machinery in \cite{FK0} or \cite[Section 2, p. 2046-2049]{FK2}, we deduce that for each integer $r\geq 0$, we have
\begin{equation}\label{eq:distribution}
|\{2d\in \DD_2:\ d\leq X/2,\ d\equiv 3\bmod 4,\ \rk_4\CL^+(8d) = r\}| \sim 2^{-r^2}\eta_{\infty}(2)\eta_{r}(2)^{-2} \cdot A(X; 3, 4),
\end{equation}
where
$$
\eta_{k}(2) = \prod_{j = 1}^k(1-2^{-j}) \quad\text{for }k = r, \infty.
$$
This is the same distribution as that for the $4$-rank of class groups of imaginary quadratic fields. Using \eqref{eq:distribution} with $r = 0$ in conjunction with Proposition~\ref{Prop2}, we conclude that 
$$
|\SD(X)| \gg \eta_{\infty}(2) A(X; 3, 4) \gg \eta_{\infty}(2)\frac{C_{2}}{6}\frac{X}{\sqrt{\log X}},
$$
as was to be shown.

\subsection*{Acknowledgements} The author is supported by ERC grant agreement No.\ 670239.

\bibliographystyle{abbrv}
\bibliography{M5_References}

\end{document}